\documentclass[12pt]{article}
\usepackage{graphicx}
\usepackage[english]{babel}
\usepackage[letterpaper,top=2.5cm,bottom=2.5cm,left=2.5cm,right=2.5cm,marginparwidth=1.75cm]{geometry}
\usepackage{amsmath,amsthm,amssymb, enumitem, mathtools, mathrsfs,bbm}
\usepackage{graphicx}
\usepackage{subcaption}
\usepackage{bm}
\usepackage[colorlinks = true,
            linkcolor = blue,
            urlcolor  = blue,
            citecolor = blue,
            anchorcolor = blue,
            hyperfootnotes=false]{hyperref}
\usepackage[symbol]{footmisc}

\usepackage[giveninits=true,backend=biber, maxnames=4,maxalphanames=4,style=alphabetic,doi=false,isbn=false,url=false,eprint=true]{biblatex}
\renewbibmacro{in:}{}
\AtEveryBibitem{%
  \clearlist{language}%
}

\newcommand{\R}{\mathbb{R}}

\renewcommand{\d}{\mathrm{d}}

\newcommand{\lb}{\left}
\newcommand{\rb}{\right}

\newcommand{\wt}{\widetilde}
\newcommand{\wh}{\widehat}

\newcommand{\bP}{\mathbb{P}}

\newcommand{\bF}{\mathbb{F}}
\newcommand{\E}{\mathbb{E}}

\newcommand{\calF}{\mathcal{F}}

\theoremstyle{plain}
\newtheorem{theorem}{Theorem}
\newtheorem*{theorem*}{Theorem}

\newtheorem*{lemma*}{Lemma}
\newtheorem{lemma}{Lemma}
\newtheorem{corollary}{Corollary}

\theoremstyle{definition}

\theoremstyle{remark}
\newtheorem{remark}{Remark}
\newtheorem*{remark*}{Remark}

\begin{document}

\begin{center}
\Large
Minimal Solutions to the Skorokhod Reflection Problem Driven by Jump Processes and an Application to Reinsurance

\vspace{1em}
\normalsize
\today

\vspace{1em}

Graeme Baker\footnote{Department of Statistics, Columbia University, NY, USA \href{mailto:g.baker@columbia.edu}{g.baker@columbia.edu}} and Ankita Chatterjee\footnote{Department of Mathematics, Barnard College, NY, USA \href{mailto:ac5481@barnard.edu}{ac5481@barnard.edu}} 

\end{center}

\begin{abstract}
We consider a reflected process in the positive orthant driven by an exogenous jump process. For a given input process, we show that there exists a unique minimal strong solution to the given particle system up until a certain stopping time, which is stated explicitly in terms of the dual formulation of a linear programming problem associated with the state of the system. We apply this model to study the ruin time of interconnected insurance firms, where the stopping time can be interpreted as the failure time of a reinsurance agreement between the firms. Our work extends the analysis of the particle system in \cite{baker_particle_2025} to the case of jump driving processes, and the existence result of \cite{reiman_open_1984} beyond the case of sub-stochastic reflection matrices.
\end{abstract}

\section{Introduction}
Fix a probability space $(\Omega,\calF,\bP)$ with a filtration $\bF=(\calF_t)_{t\ge0-}$ satisfying the usual hypotheses. Throughout this work, we require all of our stochastic processes to be c\`adl\`ag (meaning right continuous with left limits). We denote the jump of any given process $A$ at time $t$ by $\Delta A_t:=A_t-A_{t-}$. To allow for jumps at the initial time, we prepend a left limit $0-$ to a semi-infinite interval and consider the index set $\{0-\}\cup[0,\infty)$. Consider $n\ge 1$ stochastic processes $X^1,\dots,X^n$, which for $t\ge 0-$ satisfy the system of equations
\begin{equation}\label{eq:ps}
X_t^i=X_{0-}^i+c_it-Z_t^i+L_t^i-\sum_{j\neq i}q_{ij}L_t^j,\quad 1\le i\le n,
\end{equation}
where $X_{0-}^1,\dots, X_{0-}^n\ge 0$ are $\calF_{0-}$-measurable initial conditions; $c_1,\dots,c_n$ are non-negative constants; $Q=(q_{ij})_{i,j=1}^n$ is a non-negative matrix with $q_{ii}=0$ for $1\le i\le n$; $Z^1,\dots,Z^n$ are $\bF$-adapted jump processes with (almost surely) finitely many jumps in any interval (for instance, compound Poisson processes); and $L^1,\dots,L^n$ are reflection processes which constrain $X^1,\dots,X^n$ to remain in the non-negative orthant $\R^n_+$. 

\medskip

We ask that the reflection processes each satisfy a one-dimensional Skorokhod reflection problem (see, for instance, \cite[Lemma 6.14]{karatzas_brownian_2014}):
\begin{equation}\label{eq:Ldef}
L_t^i=\sup_{s\le t}\bigg(X_{0-}^i+c_it-Z_s^i-\sum_{j\neq i}q_{ij}L_s^j\bigg)_-,\quad 1\le i\le n,
\end{equation}
where $(a)_-:=-\min(0,a)$ for $a\in\R$. For a given input process $Y$ taking values in $\R$, $L_t:=\sup_{s\le t} (Y_s)_-$ is the smallest non-decreasing process such that $Y+L$ is non-negative for all times. Given inputs $X_{0-}^1,\dots, X_{0-}^n$, and $Z^1,\dots,Z^n$, a strong solution to \eqref{eq:ps}--\eqref{eq:Ldef} consists of a pair of processes $X=(X^1,\dots,X^n)$ and $L=(L^1,\dots,L^n)$ which simultaneously satisfy \eqref{eq:ps} and \eqref{eq:Ldef} on a (possibly random) time interval $[0,\tau)$. 
In Theorem \ref{thm}, we show that a strong solution to \eqref{eq:ps}--\eqref{eq:Ldef} on $[0,\tau)$ can be continued to $[0,\tau]$ if and only if $X_{\tau-}-\Delta Z_\tau$ is contained in a certain dual cone. For $x,y\in\R^n$, introduce the notation $x\ge y$ to mean that $x_i\ge y_i$ for all $1\le i\le n$. As a consequence (Corollary \ref{coro}), we establish existence and uniqueness of a \emph{minimal} solution $(X,L)$ and a \emph{maximal} stopping time $\tau^*$ such that if $(\wt{X},\wt{L})$ is any other strong solution to \eqref{eq:ps}--\eqref{eq:Ldef} on $[0,\wt\tau)$ with the same initial condition and driving processes as $(X,L)$ then $\wt\tau\le\tau^*$ and $\wt{L_t}\ge L_t$ for all $t\in[0,\tau)$.

\medskip

In a financial context, we interpret $X^1,\dots,X^n$ as the resource levels of $n$ insurance firms bound by a reinsurance agreement. The firms collect premiums at the constant rates $c_1,\dots,c_n$ and pay claims according to the exogenous shocks $Z^1,\dots,Z^n$. The matrix $Q$ encodes the routing of resources between firms due to reinsurance: firm $i$ remains solvent thanks to the the term $L^i$, but must contribute $q_{ij}L^j$ to firm $j$. The amount transferred from all firms to firm $j$ by time $t$ is $\sum_{i\neq j} q_{ij}L^j_t$. If $\sum_{i\neq j} q_{ij}>1$, there is friction in the transfers to $j$, possibly due to taxes or fees. The minimal solution $L^1,\dots,L^n$ gives the most parsimonious choice to redistribute resources between $X^1,\dots,X^n$, keeping all resource levels non-negative at all times prior to the maximal time $\tau^*$. We interpret $\tau^*$ itself as a ruin time, when there is insufficient liquidity and the reinsurance agreement breaks down. In Section \ref{sec:app}, we compare ruin probabilities for a particular case with $n=2$, both with and without reinsurance.

\medskip

\section{Main Result}

When does \eqref{eq:ps}--\eqref{eq:Ldef} admit solutions? If the spectral radius of $Q$, $\rho(Q)$, is strictly less than 1, then solutions exist for all time by \cite[Proposition 1]{reiman_open_1984} (and see also \cite{harrison_reflected_1981}, for continuous driving processes such as Brownian motion). For $\rho(Q)\ge 1$, can we go beyond this result? When the driving processes are Brownian motions, recent work has shown that existence and uniqueness hold on a stochastic interval which depends on the structure of $Q$, the covariance of the driving processes, and the given initial condition \cite{baker_particle_2025}. For the present setting where $X$ is driven by jump processes we state our main result, which we prove in Section \ref{sec:proof}. 

\begin{theorem}\label{thm}
Consider the cone $C=\{u\in\R^n:u\ge 0, u^\top (I-Q)\le 0\}$ and its dual cone $C^*=\{y\in\R^n:u^\top y \ge 0 \text{ for all }u\in C\}$.
A solution $(X,L)$ of \eqref{eq:ps}--\eqref{eq:Ldef} on an interval $[0,\tau)$ can be extended to $[0,\tau]$ if and only if $X_{\tau-}-\Delta Z_\tau \in C^*$. Moreover, when $X_{\tau-}-\Delta Z_\tau \in C^*$ there exists a unique minimal $\Delta L_\tau$ with respect to the partial order $\ge$ on $\R_+^n$ such that $L_\tau=L_{\tau-}+\Delta L_\tau$ gives an extension of $(X,L)$ on $[0,\tau]$.
\end{theorem}

As a corollary, show that for any given initial conditions and driving processes $(Z^1,\dots,Z^n)$ there exists a unique minimal strong solution of \eqref{eq:ps}--\eqref{eq:Ldef} on a maximal time interval (see, for instance, \cite{delarue_particle_2015,cuchiero_propagation_2023,baker_particle_2025} for similar notions in a range of contexts).

\begin{corollary}\label{coro}
Given inputs $X_{0-}^1,\dots, X_{0-}^n$, and $Z^1,\dots,Z^n$, there exists unique $(X,L)$ and $\tau^*$ such that $(X,L)$ solves \eqref{eq:ps}--\eqref{eq:Ldef} on $[0,\tau^*)$ and if $(\wt{X},\wt{L})$ is any other strong solution to \eqref{eq:ps}--\eqref{eq:Ldef} on $[0,\wt\tau)$ with the same inputs as $(X,L)$ then $\wt\tau\le\tau^*$ and $\wt{L_t}\ge L_t$ for all $t\in[0,\wt\tau)$.
\end{corollary}

\begin{remark}
If $\rho(Q)<1$, then it is easy to see that $C$ is empty, $C^*=\R^n$, and hence solutions may always be continued. Therefore, we have found a new derivation for the existence result of \cite[Proposition 1]{reiman_open_1984}.
\end{remark}

\begin{proof}[Proof of Corollary \ref{coro}]
Consider any arbitrary solution  $(\wt{X},\wt{L})$ to \eqref{eq:ps}--\eqref{eq:Ldef} on some interval $[0,\wt\tau)$.
Let $\tau_1$ denote the first time that one or multiple of the processes $Z^1,\dots,Z^n$ exhibit a jump. Set $L^1_t=\dots =L^n_t=0$ on $[0,\tau_1)$. From \eqref{eq:Ldef}, we see that for $1\le i\le n$, $\wt{L}^i_t\ge 0=L^i_t$ for all $t\in [0,\tau_1\wedge \wt\tau)$, where we have used the notation $s\wedge t:=\min(s,t)$. We consider now two cases:

\medskip

\emph{Case 1.} If $X_{\tau_1-}-\Delta Z_{\tau_1} \notin C^*$, then we set $\tau^*=\tau_1$. If $\wt\tau \ge \tau^*$ then we must have $\wt{X}_{\tau_1-}-\Delta Z_{\tau_1} \notin C^*$ and hence $\wt\tau > \tau^*$ is impossible.

\medskip

\emph{Case 2.} If $X_{\tau_1-}-\Delta Z_{\tau_1} \in C^*$ we let $\Delta L_{\tau_1}$ be the unique minimal jump size from Theorem \ref{thm} and set $L_{\tau_1}=0+\Delta L_{\tau_1}$.
If $\wt\tau < \tau_1$, then $\wt{L}^i_t\ge L^i_t$ for all $t\in [0,\tau_1\wedge \wt\tau]$ trivially. If not, then suppose for the sake of contradiction that $\wt{L}^i_{\tau_1}< L^i_{\tau_1}$ for some $1\le i\le n$ and for $1\le i \le n$ set 
\begin{equation*}
\wh{L}^i_t=\begin{cases}
0&t<\tau_1\\
\wt{L}^i_{\tau_1} & t=\tau_1.
\end{cases}
\end{equation*}
Then, the definition of $\tau_1$ along with the non-negativity of $\wt{X}$ imply that
\begin{equation*}
X_{0-}^i+c_it-Z^i_t+\wh{L}^i_t-\sum_{j\neq i}q_{ij} \wh{L}^j_t\ge 0
\end{equation*}
for all $t\in[0,\tau_1]$ and $1\le i\le n$. However, $(\Delta\wh{L}^1_{\tau_1},\dots,\Delta\wh{L}^n_{\tau_1})=(\wt{L}^1_{\tau_1},\dots, \wt{L}^n_{\tau_1})$ contradicts the minimality of $\Delta L_{\tau_1}$ with respect to the partial order $\ge$ on $\R_+^n$. Therefore $\wt{L}^i_t\ge L^i_t$ for all $t\in [0,\tau_1\wedge \wt\tau]$. 

\medskip

Let $\tau_2$ denote the next jump time for any of the processes $Z^1,\dots,Z^n$. Repeating the above argument shows that $\wt{L}^i_t\ge L^i_t$ for all $t\in [0,\tau_2\wedge \wt\tau]$ and $\wt\tau > \tau^*$ is impossible. We continue as before with $\tau_3,\tau_4$, and so on. Since $(\wt{X},\wt{L})$ and $\wt\tau$ are arbitrary, we have completed the proof. \qedhere
\end{proof}

\section{Increments of the Skorokhod Map}

We prove here an auxiliary result, which allows us to express $\Delta L^1_\tau,\dots, \Delta L^n_\tau$ in terms of $X_{\tau-}$ and $\Delta Z_\tau$.

\begin{lemma}\label{lemma}
Let $(Y_t)_{t\ge 0-}$ be a c\`adl\`ag process taking values in $\R$. Define $L_t=\sup_{s\le t} (Y_s)_-$ and $X_t=Y_t+L_t$. Then for all $t\ge 0-$,
\begin{equation}\label{eq:lemma}
\Delta L_t =\lb(X_{t-}+\Delta Y_t\rb)_-
\end{equation}
\end{lemma}
\begin{proof}
We split the proof into the cases $-Y_t\le L_{t-}$ and $-Y_t> L_{t-}$
First, suppose $-Y_t\le L_{t-}$. Then $\Delta L_t=0$ and 
\begin{align*}
X_{t-}+\Delta Y_t=Y_{t-}+L_{t-}+Y_t-Y_{t-}\ge 0.
\end{align*}
Hence $\lb(X_{t-}+\Delta Y_t\rb)_-=0=\Delta L_t$ and \eqref{eq:lemma} holds. Next, if $-Y_t> L_{t-}$, then $L_t=-Y_t\ge 0$ and hence
\begin{align*}
\Delta L_t = L_t-L_{t-}=-Y_t-(X_{t-}-Y_{t-})=-(X_{t-}+\Delta Y_t)\ge 0.
\end{align*}
We see that \eqref{eq:lemma} is established in this case as well. \qedhere
\end{proof}

\section{Proof of Theorem \ref{thm}}\label{sec:proof}

We proceed by reducing to a linear programming problem. 
An example of the primal and dual problems with $n=2$ is plotted in Figure \ref{fig:plots}.

\begin{proof} Suppose $\Delta Z_\tau=(\Delta Z^1_\tau,\dots,\Delta Z^n_\tau)$ is non-zero (otherwise, the problem is trivial). By Lemma \ref{lemma}, if we can exhibit a jump $\Delta L_\tau=(\Delta L^1_\tau,\dots,\Delta L^n_\tau)\in\R^n_+$ such that
\begin{equation}\label{eq:Ljump}
\Delta L^i_\tau=\bigg(X_{\tau-}^i-\Delta Z_\tau^i-\sum_{j\neq i}q_{ij}\Delta L_\tau^j\bigg)_-,\quad 1\le i\le n,
\end{equation}
then setting $L_\tau=L_{\tau_-}+\Delta L_\tau$ and $X_\tau=X_{\tau_-}-\Delta Z_\tau+(I-Q)\Delta L_\tau$ yields a solution to \eqref{eq:ps}--\eqref{eq:Ldef} on $[0,\tau]$. Consider the linear programming (LP) problem
\begin{align*}
\text{minimize} &\quad\|\Delta L_\tau\|_1\\
\text{subject to} &\quad\Delta L_\tau \ge 0\text{ and }(I-Q)\Delta L_{\tau}\ge -\lb(X_{\tau-}-\Delta Z_\tau\rb).
\end{align*}
If $\Delta L_\tau$ solves LP, then for each $i$, at least one of the following equalities must hold
\begin{align*}
\Delta L^i_\tau=0\quad\text{ or }\quad\Delta L^i_\tau=-\lb(X_{\tau-}^i-\Delta Z_\tau^i\rb)+\sum_{j\neq i}q_{ij}\Delta L_\tau^j,
\end{align*}
since otherwise $\Delta L^i_\tau$ may be decreased and the objective function $\|\Delta L_\tau\|_1$ will be decreased. Hence, if $\Delta L_\tau$ solves LP then \eqref{eq:Ljump} holds. We tackle the feasibility of LP using duality.

\medskip

Rewrite the constraints on $\Delta L_\tau$ as
\begin{equation*}
A\Delta L_\tau\ge b \quad \text{where}\quad  A= \begin{bmatrix}
    I-Q \\
    I
\end{bmatrix} 
\quad\text{and}\quad b = 
\begin{bmatrix}
    -X_{\tau-}+\Delta Z_\tau \\
    0
\end{bmatrix}.
\end{equation*}
Farkas's Lemma (see, for instance \cite[Theorem 2.1]{dantzig_linear_1997-1}, where it is called the Infeasibility Theorem) implies that LP is infeasible if and only if there exists $y\ge 0$ such that $y^\top A =0$ and $y^\top b >0$. We seek to reduce this latter condition to the condition in the statement of the theorem. Write $y^\top=[u^\top,\, v^\top]$ where $u,v\in \R^n$. Then $y^\top A =0$ implies that $v^\top=-u^\top(I-Q)$. The condition $y^\top b >0$ simplifies to
\begin{equation*}
y^\top b=-u^\top(X_{\tau-}-\Delta Z_\tau)>0.
\end{equation*}
Therefore, LP is infeasible if and only if there exists $u\ge 0$ with $-u^\top(I-Q)=v^\top\ge 0$ and $-u^\top(X_{\tau-}-\Delta Z_\tau)>0$. The conditions on $u$ and $v$ define the cone $C$ from the statement of the theorem, and there will be $u\in C$ such that $-u^\top(X_{\tau-}-\Delta N_\tau)>0$ iff $X_{\tau-}-\Delta N_\tau\notin C^*$. 

\medskip

It remains to show uniqueness of the minimal jump. Suppose that $\Delta L_\tau\ge 0$ and ${[I-Q]_i \Delta L_\tau =0}$, where $[I-Q]_i$ denotes the $i$th row of $I-Q$. Then $\Delta L_{\tau}^i=\sum_{j\neq i}q_{ij}\Delta L_{\tau}^j$ and hence
\begin{equation*}
(1,1,\dots,1)^\top \Delta L_{\tau}=\Delta L_{\tau}^i+\sum_{j\neq i}\Delta L_{\tau}^j=\sum_{j\neq i}(1+q_{ij})\Delta L_{\tau}^j\ge 0
\end{equation*}
with equality if and only if $\Delta L_\tau=0$. Therefore, none of the constraining faces in the feasible region are orthogonal to $(1,1,\dots,1)$, and hence uniqueness holds for LP. Note that uniqueness still holds if the objective function is replaced by $a^\top \Delta L_\tau$ for any $a$ with strictly positive entries. This yields minimality with respect to the partial order $\ge$ on $\R_+^n$.
\qedhere
\end{proof}

\begin{figure}[h]
\centering
\begin{subfigure}{0.45\textwidth}
\includegraphics[width=1\linewidth]{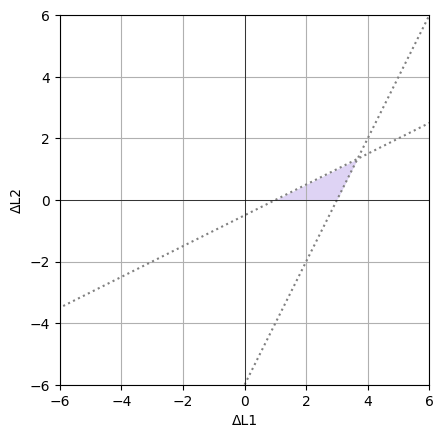}
\caption{Primal problem. The feasible region is plotted in purple. The minimal jump is given by $(1,0)$.}
\end{subfigure}
~
\begin{subfigure}{0.45\textwidth}
\includegraphics[width=1\linewidth]{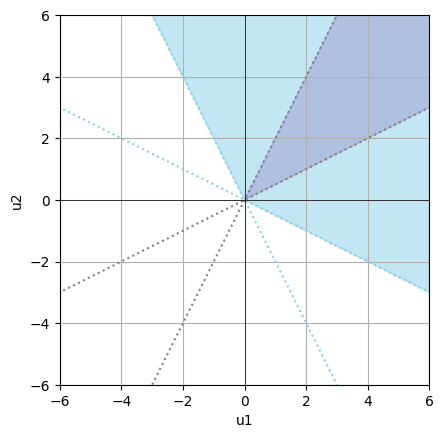}
\caption{Dual problem. $C$ is plotted in purple, and the union of the purple and blue regions gives $C^*$. Note that $(-1,6)\in C^*$.}
\end{subfigure}%    
\caption{Example primal and dual problems with $n=2$, $q_{12}=q_{21}=2$, and ${(X_{\tau-}^1-\Delta Z_{\tau}^1,X_{\tau-}^2-\Delta Z_{\tau}^2)=(-1,6)}$. In this case, the primal problem is feasible.}
\label{fig:plots}
\end{figure}

\section{Fixed Point Approach}

When the LP problem is feasible, it can be solved efficiently using the simplex method \cite{dantzig_linear_1997-1}. Another approach is to construct a solution iteratively with a monotone operator (see, for instance, a similar approach in \cite[Proposition 2.3]{cuchiero_propagation_2023}). For a given input vector $X_{\tau-}-\Delta Z_\tau$, define the operator $\Gamma:\R_+^n\to\R_+^N$ by
\begin{equation*}
\Gamma[z]=\bigg(\big(X_{\tau-}^1-\Delta Z_\tau^1-\sum_{j\neq 1} q_{1j} z_j\big)_-,\dots,\big(X_{\tau-}^1-\Delta Z_\tau^1-\sum_{j\neq n} q_{nj} z_j\big)_-\bigg)
\end{equation*}
We notice that fixed points of $\Gamma$ (that is, any $z^*$ such that $\Gamma[z^*]=z^*$) satisfy \eqref{eq:Ljump}. One can check that $\Gamma$ is monotone (that is, $\Gamma[z]\ge z$ for all $z$). Therefore, the Knaster--Tarski Theorem implies that the fixed points of $\Gamma$ form a complete lattice (this result is often used for existence of fixed points in systemic risk literature, such as the seminal work \cite{eisenberg_systemic_2001}). Starting from $z=0$ and applying $\Gamma$ iteratively yields the (\emph{a fortiori}) unique solution to LP in the limit. 

\section{Application to Reinsurance}\label{sec:app}

We consider now a particular case with $n=2$ insurance firms. As a base case in the absence of reinsurance, we suppose that for $t\ge 0-$, $X^{(1)}$ and $X^{(2)}$ satisfy 
\begin{equation*}
X^{(1)}_t=X^{1}_{0-}+c_1 t-Z^1_t,\quad\text{and}\quad
X^{(2)}_t=X^{2}_{0-}+c_1 t-Z^2_t,
\end{equation*}
where
\begin{equation*}
Z^1_t=\sum_{k=1}^{N^1_t+N^3_t} Y^1_k,\quad\text{and}\quad
Z^2_t=\sum_{k=1}^{N^2_t+N^3_t} Y^2_k;
\end{equation*}
$X^1_{0-},X^2_{0-}$ are $\calF_{0-}$-measurable initial conditions; $c_1,c_2\ge0$ are fixed; $N^1,N^2$, and $N^3$ are independent $\bF$-adapted Poisson processes with intensities $\lambda_1,\lambda_2$, and $\lambda_3$; and $(Y^1_k)_{k\ge1}$ and $(Y^2_k)_{k\ge1}$ are each sequences of independent and identically distributed random variables with distributions $Y^1_1$ and $Y^2_1$, respectively. The ruin time of each firm is given by
\begin{equation*}
T^1=\inf\{t\ge0: X^{(1)}_t\le 0\},\quad\text{and}\quad T^2=\inf\{t\ge0: X^{(2)}_t\le 0\}.
\end{equation*}
We have here two instances of a classical insurance risk model (see, for instance, \cite{kluppelberg_ruin_2004}), which have been coupled together by the common Poisson process $N^3$.

\medskip

In Figure \ref{fig:failureprob}, we plot Monte Carlos estimates of the ruin probabilities $t\mapsto \bP(T^1\le t)$, $t\mapsto \bP(T^2\le t)$, $t\mapsto \bP(T^1\le t, T^2\le t)$, and $t\mapsto \bP(T^1\le t \text{ or } T^2\le t)$ for a particular choice of parameters. The plots were generated using 20000 Monte Carlo trials, $X^{(1)}_{0-} = X^{(2)}_{0-} = 5$, $Y^1_1$ and $Y^2_1$ are both $\mathrm{Exp}(1)$ random variables,  $c_1 = c_2 = 1.2$, and $\lambda_1 = \lambda_2 = \lambda_3 = 0.6$. The conditions for the firms are symmetrical so that $\bP(T^1\le t)=\bP(T^2\le t)$ for all $t$. Closed form solutions for the ruin probabilities are available in some special cases (see \cite{kluppelberg_ruin_2004} and also \cite{bernyk_law_2008}), but we do not pursue them here. For limiting behavior as $t\to\infty$, \cite[Lemma 8.7.1.1]{jeanblanc_mathematical_2009} gives that $\bP(T^1<\infty)=1$ if $c_1/(\lambda_1+\lambda_3)\le \E[Y_1^1]<\infty$, and similarly for $T^2$. More generally, one may deduce an integral equation for ${\Psi(x):=\bP(T^1=\infty|X^1_{0-}=x)}$ in the variable $x$ (we refer the interested reader to \cite[Subsection 8.7.2]{jeanblanc_mathematical_2009}).

\begin{figure}
\centering
\includegraphics[width=0.8\linewidth]{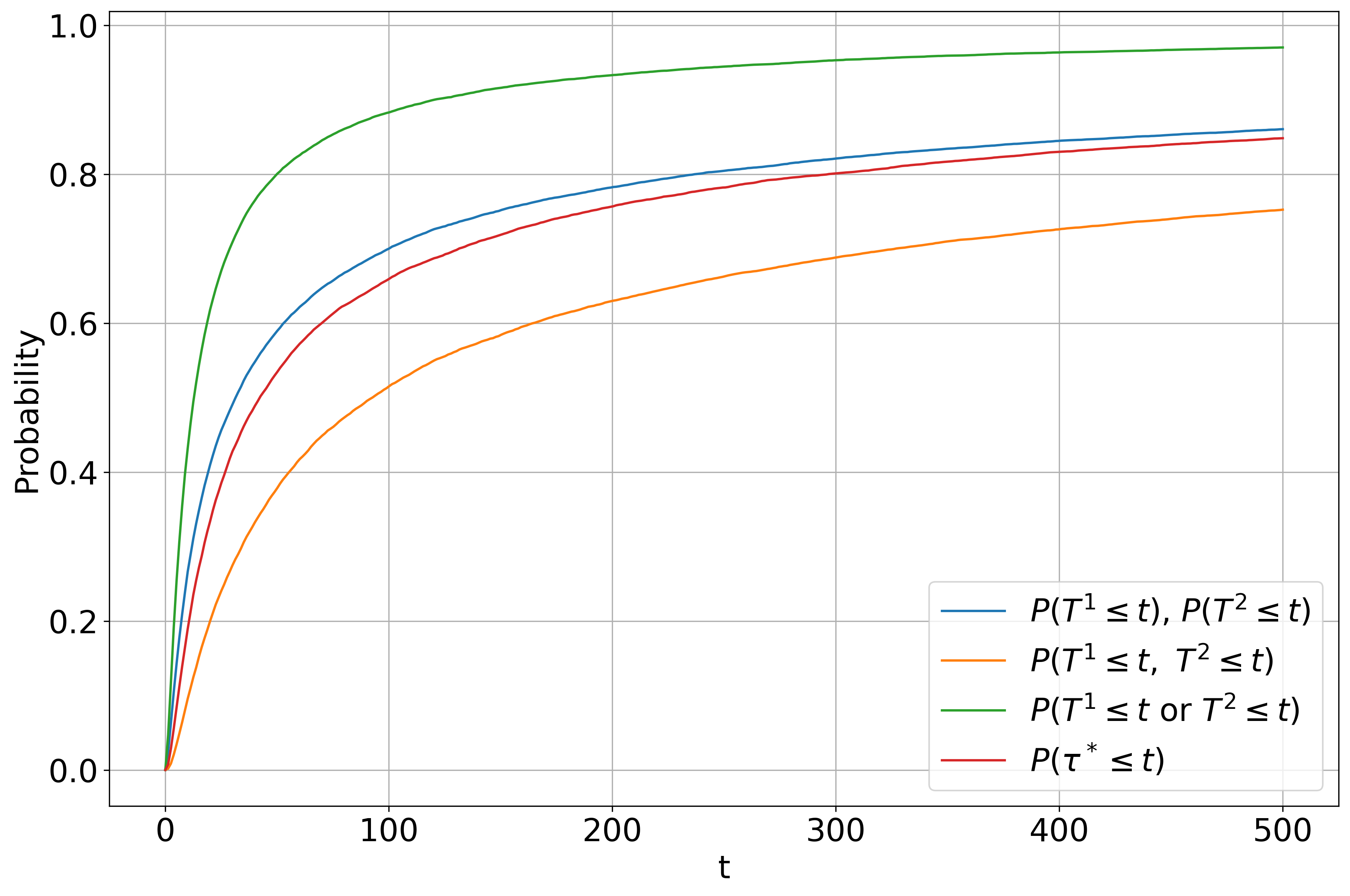}
\caption{Ruin probabilities on the time interval $t\in[0,500]$ when $\alpha=0.05$.}
\label{fig:failureprob}
\end{figure}

\medskip

Next, we compare to the case with reinsurance, that is, the minimal solution of the reflected system \eqref{eq:ps}--\eqref{eq:Ldef}. We take the same initial conditions and driving processes. We suppose that $q_{12}=q_{21}=1+\alpha$ where $\alpha>0$ represents friction due to taxes and fees. In Figure \ref{fig:failureprob}, we have plotted a Monte Carlo estimate for $t\mapsto \bP(\tau^*\le t)$ (again using 20000 trials) alongside $t\mapsto \bP(T^1\le t)$,  $t\mapsto \bP(T^1\le t, T^2\le t)$, and $t\mapsto \bP(T^1\le t \text{ or } T^2\le t)$, with $\alpha=0.05$ and all other parameters unchanged. In general, the distribution of $\bP(\tau^*\le t)$ is difficult to compute. However, for the special case $\alpha=0$ we see that 
\begin{equation*}
X^1_t+X^2_t=X^{1}_{0-}+X^{2}_{0-}+(c_1 +c_2)t-Z^1_t-Z^2_t,
\end{equation*}
is also a classical 1-dimensional ruin model. Using \cite[Lemma 8.7.1.1]{jeanblanc_mathematical_2009}, we see that $\bP(\tau^*<\infty)=1$ when $(c_1+c_2)/(\lambda_1+\lambda_2+\lambda_3)\le \frac{2}{3}(\E[Y_1^1]+\E[Y_2^1])$, and by comparison this sufficient condition holds for any $\alpha>0$. 

\medskip

With the parameters used in Figure \ref{fig:failureprob}, at any given time we see that the Monte Carlo estimate for the ruin probability of the reflected system $(X^1,X^2)$ is higher than the probability that both $X^{(1)}$ and $X^{(2)}$ fail, but lower than the probability that at least one of $X^{(1)}$ and $X^{(2)}$ fail. Furthermore, $\bP(\tau^*\le t)\le \bP(T^1\le t)=\bP(T^2\le t)$ for all $t$. This suggests that each individual firm might consider opting into this reinsurance agreement to increase their individual chance of survival on a given time horizon. If the friction parameter $\alpha$ is increased, then the curve $t\mapsto \bP(T^1\le t)$ may no longer dominate $t\mapsto\bP(\tau^*\le t)$. For instance, this is observed when $\alpha$ is increased to $0.5$ and all other parameters are unchanged (see Figure \ref{fig:friction}).

\medskip

\begin{figure}
\centering
\includegraphics[width=0.8\linewidth]{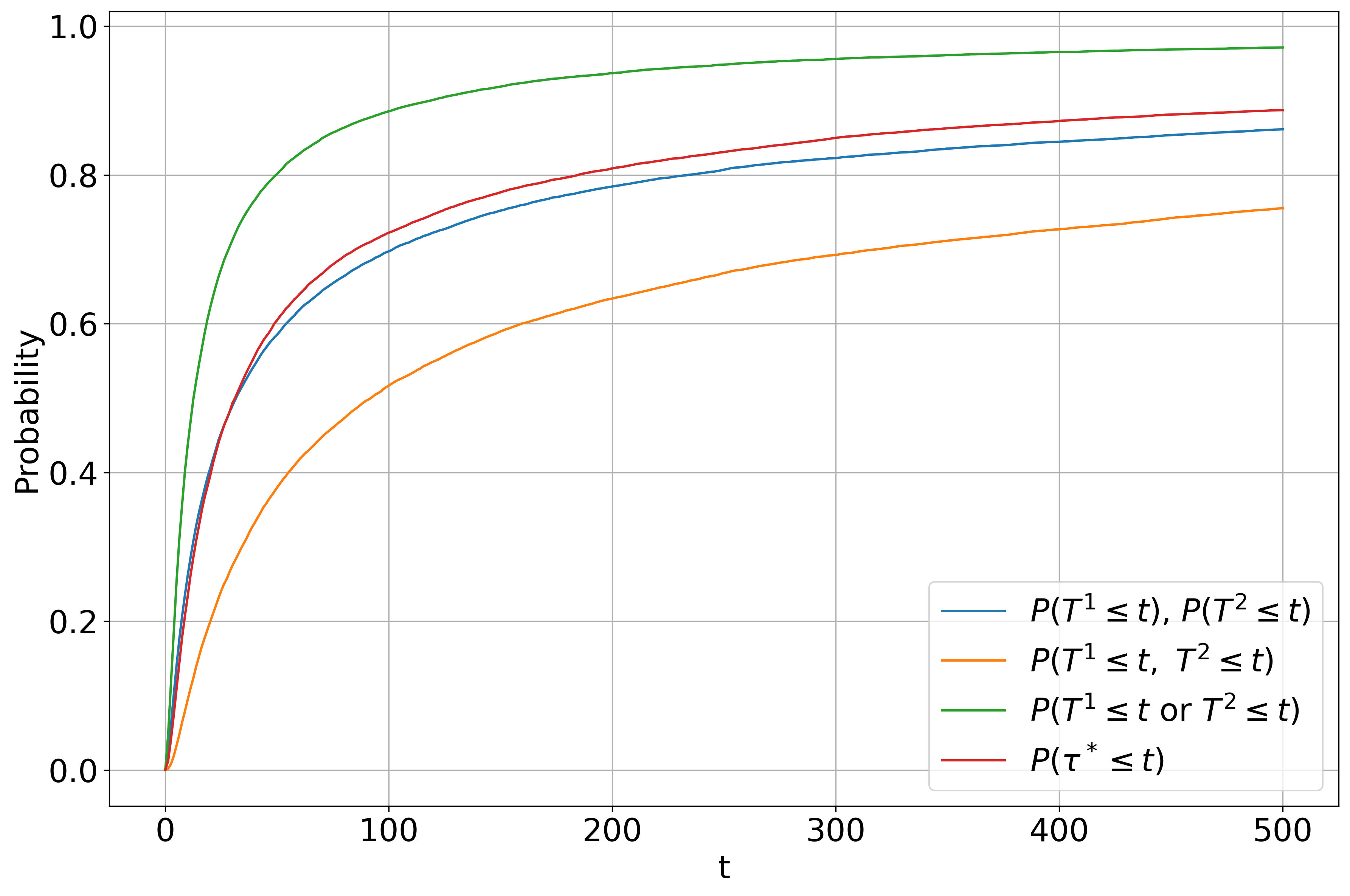}
\caption{Ruin probabilities on the time interval $t\in[0,500]$ when $\alpha=0.5$.}
\label{fig:friction}
\end{figure}

For small times, we can compare the failure probabilities by computing the slopes of the curves. The quantities $\frac{\d}{\d t} \bP(T^1\le t)|_{t=0}$, $\frac{\d}{\d t} \bP(T^2\le t)|_{t=0}$, $\frac{\d}{\d t} \bP(T^1\le t, T^2\le t)|_{t=0}$, $\frac{\d}{\d t}\bP(T^1\le t\text{ or } T^2\le t)|_{t=0}$, and $\frac{\d}{\d t} \bP(\tau^*\le t)|_{t=0}$ are given by the intensities for which $t\mapsto X_{0-}-Z_t$ jumps out of the regions $H^1=\{(y_1,y_2):y_1>0\}$, $H^2=\{(y_1,y_2):y_2>0\}$, $H^1\cup H^2$, $H^1\cap H^2$, and $C^*$, respectively. For $\alpha>0$, the inclusions $H^1\cap H^2\subseteq C^* \subseteq H^1\cup H^2$ imply that $\frac{\d}{\d t} \bP(T^1\le t, T^2\le t)|_{t=0}\le \frac{\d}{\d t} \bP(\tau^*\le t)|_{t=0} \le \bP(T^1\le t\text{ or } T^2\le t)$ for any choice of the model parameters $(X^{(1)}_{0-},X^{(2)}_{0-},Y^1_1,Y^2_1,c_1,c_2,\lambda_1, \lambda_2 , \lambda_3)$. For instance, in Figure \ref{fig:origincomparison}, we have plotted the same curves as in Figure \ref{fig:friction} on an interval near the origin: $t\in [0,10]$. While $\bP(\tau^*\le t)$ may be greater than $\bP(T^1\le t)$ for $t$ sufficiently large (Figure \ref{fig:friction}), this is not the case for small $t$ (Figure \ref{fig:origincomparison}).

\begin{figure}
\centering
\includegraphics[width=0.8\linewidth]{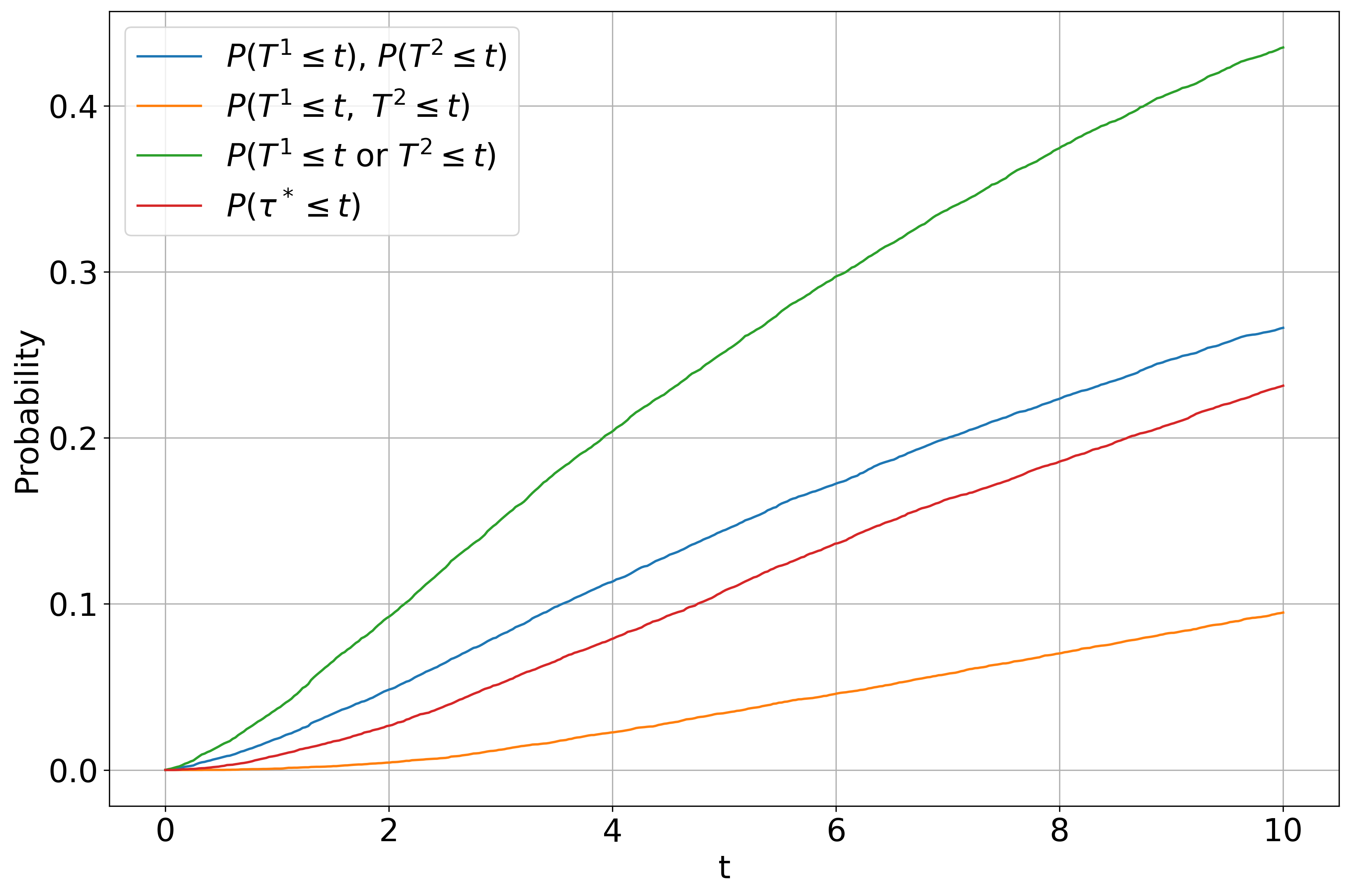}
\caption{Ruin probabilities on the time interval $t\in[0,10]$ when $\alpha=0.5$.}
\label{fig:origincomparison}
\end{figure}

\section*{Acknowledgments}

We acknowledge the Summer Research Institute at Barnard College as well as the Statistics Department at Columbia University for funding which supported this summer undergraduate research project. We thank Professor Karatzas for suggesting this collaboration, and for his helpful feedback on an early draft. Additionally, we thank Professor Reiman for many helpful discussions.

\printbibliography

\end{document}